\documentclass[reqno]{amsart}
\usepackage{amsmath,amssymb,amsthm,bm}

\def\bbbf{\mathbb F}
\def\bbbz{\mathbb Z}

\def\bbbp{\mathbb P}

\def\leg#1#2{\left({#1\over#2}\right)}

\def\is{\equiv}
\def\mod#1{({\rm mod}\ #1)}

\def\tr{{\rm Tr}}

\newtheorem{theorem}[subsection]{Theorem}

\newtheorem{lemma}[subsection]{Lemma}

\newtheorem{corollary}[subsection]{Corollary}

\newtheorem{conjecture}[subsection]{Conjecture}

\begin{document}
\title[]{Some supercongruences of arbitary length}

\author[Frits Beukers]{Frits Beukers}
\address{Utrecht University, Department of Mathematics, P.O. Box 80.010,
3508 TA Utrecht, Netherlands}
\email{f.beukers@uu.nl}
\urladdr{http://www.staff.science.uu.nl/{\textasciitilde}beuke106/}

\author[Eric Delaygue]{Eric Delaygue}
\address{Univ Lyon,, Universit\'e Claude Bernard Lyon 1, CNRS UMR 5208, Institut
Camille Jordan, F-69622 Villeurbanne, France}
\email{delaygue@math.univ-lyon1.fr}
\urladdr{http://math.univ-lyon1.fr/{\textasciitilde}delaygue/}

\thanks{The authors would like to thank the Matrix Institute and the organizers
of the workshop Hypergeometric Motives and Calabi-Yau Differential Equations for
the wonderful and stimulating environment in which this research arose.
This project has received funding from the European Research Council (ERC) under the European Union's Horizon 2020 research and innovation programme 
under the Grant Agreement No 648132.}

\date{}



\begin{abstract}
We prove supercongruences modulo $p^2$ for values of truncated hypergeometric
series at some special points. The parameters of the hypergeometric series are $d$
copies of $1/2$ and $d$ copies of $1$ for any integer $d\ge2$. In addition we
describe their relation to hypergeometric motives.
\end{abstract}

\maketitle

\section{Introduction}
Fix an integer $d\ge2$ and consider the hypergeometric series
$$F(z)=\sum_{n=0}^{\infty} \left({(1/2)_n\over n!}\right)^d z^n,$$
where $(x)_n$ denotes the product $x(x+1)(x+2)\cdots(x+n-1)$. It is known as the
Pochhammer symbol. Let $p$ be a fixed odd prime. For every integer $s\ge0$ we define
the truncated series
$$
F_{p^s}(z)=\sum_{n=0}^{p^s-1} \left({(1/2)_n\over n!}\right)^d z^n.
$$
In particular $F_1(z)=1$.
Let $z_0$ be a $p$-adic unit and suppose that $F_p(z_0)$ is also a $p$-adic
unit. Then, by a result of Dwork \cite{DworkIV}, we have for all $s\ge1$ 
that $F_{p^s}(z_0)$ is a $p$-adic unit together with the congruence
\begin{equation}\label{cauchy}
{F_{p^{s+1}}(z_0)\over F_{p^s}(z_0)}\is {F_{p^s}(z_0)\over F_{p^{s-1}}(z_0)}\mod{p^s}.
\end{equation}
So the sequence of quotients is a $p$-adic Cauchy sequence. We define the limit
$$f(z_0)=\lim_{s\to\infty}{F_{p^s}(z_0)\over F_{p^{s-1}}(z_0)}.$$
The number $f(z_0)$ is refered to as the {\it unit root part}
of the Frobenius-action on a suitable $p$-adic cohomology. We shall make this a
bit more explicit in Section \ref{motive}.

From (\ref{cauchy}) it follows that $f(z_0)\is F_p(z_0)\mod{p}$. But it turns out that
for some values of
$z_0$ one has stronger congruences, a remarkable phenomenon called {\it supercongruences}.
In this paper we prove the following theorem,

\begin{theorem}\label{main}
Let $\epsilon_p=(-1)^{d(p-1)/2}$ and suppose that $F_p(\epsilon_p)$ is a $p$-adic unit.
Then
$$F_p(\epsilon_p)\is f(\epsilon_p)\mod{p^2}.$$
\end{theorem}

This proves part of the following conjecture we like to propose here.

\begin{conjecture}\label{conj1}
With the notations as above let $\epsilon=\pm1$ and suppose that $F_p(\epsilon)$
is a $p$-adic unit. Then $F_p(\epsilon)\is f_p(\epsilon)\mod{p^2}$
\end{conjecture}

It should be remarked that if $p\is3\mod{4}$ then $F_p(-\epsilon_p)\is0\mod{p}$
by Corollary \ref{Fzero}. So the only values $F_p(\epsilon)$ which are still conjectural
are $F_p(-1)$ with $p\is1\mod{4}$.

For some choices of $d,\epsilon$ we conjecture some stronger congruences.

\begin{conjecture}\label{conj2}
Suppose that $F_p(\epsilon)$ is a $p$-adic unit. Then we have $F_p(\epsilon)
\is f_p(\epsilon)\mod{p^3}$ in the following cases: $d=3$ and $\epsilon=\pm1$,
$d=4$ and $\epsilon=1$, $d=5$ and $\epsilon=1$, $d=6$ and $\epsilon=1$.

Moreover, in the latter case we expect $F_p(1)\is f_p(1)\mod{p^5}$.  
\end{conjecture}

There are a number of results which go into this direction, although the formulation
does not contain the unit root $f_p(\epsilon)$ but an integer number, usually the
$p$-th coefficient of an $L$-series that occurs in number theory. For example, when $d=2$
Mortenson \cite{Mor1} showed that $F_p(1)\is \leg{-4}{p}\mod{p^2}$. Presumably we
have $f_p(1)=\leg{-4}{p}$. In general we expect that $f_p(\epsilon)$ is a zero of
the $p$-th factor of the L-series associated to the underlying hypergeometric motive. We explain this more in detail in section \ref{motive}

In the case $d=3$ several authors (Ishikawa, Van Hamme, Ahlgren)
independently proved that 
$$F_p(1)\is c_p\mod{p^2}$$
where $c_p$ is the $p$-th coefficient of 
$\eta(4\tau)^6\in S_3(16,\chi(-4))$, see \cite[p322]{Mor2} and the references
therein. The notation $S_k(N,\chi)$ stands for
the modular cusp forms of weight $k$ with group $\Gamma_0(N)$ and character $\chi$.
In particular $\chi(a)$ stands for the Legendre symbol $\leg{a}{.}$.
It is a CM form given by $c_p=2(a^2-b^2)$ where $p=a^2+b^2$ with $a$ odd.
For a proof we refer to \cite[Thm 4]{Mor2}. Numerical experiment shows that these
congruences do not hold modulo $p^3$. Surprisingly enough, these experiments also suggest
that $F_p(1)\is f_p(1)\mod{p^3}$. Presumably $f_p(1)$ is the unit root of
$x^2-c_px+p^2$ corresponding to the local Euler factor of the $L$-series of
the modular form. 

Kilbourn \cite{Kil} has shown that when $d=4$ we have
$$F_p(1)\is a_p\mod{p^3},$$
where $a_p$ is the coefficient of the
modular form $\eta(2\tau)^4\eta(4\tau)^4=\sum_{n\ge1}a_nq^n,\ q=e^{2\pi i\tau}$
in $S_4(8,\chi_0)$. By $\chi_0$ we denote the trivial character. Presumably
$f_p(1)$ is the $p$-adic unit root of $x^2-a_px+p^3$ corresponding to the
local Eulerfactor at $p$ of the L-series of the modular form.
We cannot prove this, but if true it implies that $f_p(1)\is a_p\mod{p^3}$.

Recently Osburn, Straub, Zudilin \cite{OSZ} proved that $F_p(1)\is b_p\mod{p^3},$
where $b_p$ is the $p$-th coefficient of the unique newform in $S_6(8,\chi_0)$.
It is conjectured that this congruence holds modulo $p^5$ for all odd $p$. 
We believe that $f_p(1)$ is the $p$-adic unit zero of $x^2-b_px+p^5$. 
Similarly as before this would imply that $f_p(1)\is b_p\mod{p^5}$.

Beside these results we
like to record the following conjecture. 

\begin{conjecture}\label{conj3}
We make the implicit assumption that $F_p(-1)$ is a $p$-adic unit.

When $d=3$ we expect $F_p(-1)\is c_p\mod{p^2}$ where $c_p$ is the $p$-th
coefficient of $\eta(\tau)^2\eta(2\tau)\eta(4\tau)\eta(8\tau)^2\in S_3(8,\chi(-8))$.
It is a CM-form with coefficients given by $2(2b^2-a^2)$ where $p=a^2+2b^2$ in case $p\is1,3
\mod{8}$. As conjectured in Conjecture \ref{conj2} we also expect that
$F_p(-1)\is f_p(-1)\mod{p^3}$.

When $d=5$ we expect $F_p(-1)\is d_p\mod{p^2}$ where $d_p=\left({-8\over p}\right)
(\delta_p^2-2p^2)$ and $\delta_p$ is the $p$-th coefficient of the form $g\in S_3(256,\chi(-4))$
whose expansion starts with
\begin{eqnarray*}
g(\tau)&=&q-2\sqrt{-2}q^3+4q^5+8\sqrt{-2}q^7+q^9+10\sqrt{-2}q^{11}+20q^{13}-8\sqrt{-2}q^{15}\\
&&-10q^{17}-10\sqrt{-2}q^{19}+32q^{21}-8\sqrt{-2}q^{23}+9q^{25}-20\sqrt{-2}q^{27}+20q^{29}+\cdots
\end{eqnarray*}
Since $f_p(-1)$ is (presumably) a zero of $x^2-d_px+p^4$ we should have $f_p(-1)\is d_p\mod{p^4}$.
However, experiment shows that $F_p(-1)\is f_p(-1)$ only holds modulo $p^2$. 
We are indebted to
Wadim Zudilin and Dave Roberts for the (conjectural) identification of the coefficients $d_p$.
\end{conjecture}

A natural, and often asked question, is what happens with the values of 
$F_{p^s}(\epsilon)$ with $\epsilon=\pm1$. Numerical experiment suggests the following 
generalization of Theorem \ref{main} might be true.

\begin{conjecture}
Let $\epsilon=\pm1$ and suppose that $F_p(\epsilon)$ is a $p$-adic unit. Then we have
$$F_{p^s}(\epsilon)\is f_p(\epsilon)F_{p^{s-1}}(\epsilon)\mod{p^{2s}}$$
for all integers $s\ge 1$. 
\end{conjecture}

Beside supercongruences for hypergeometric sums with parameters $1/2$ and $1$
there exist several other types for other parameter choices. We refer
to \cite{LTYZ} for a proof of Rodriguez-Villegas's mod $p^3$ conjecture
for the 14 truncated hypergeometric sums of order $4$ corresponding to
Calabi-Yau varieties.

The key to the proof of Theorem \ref{main} is the special
symmetry of the hypergeometric differential equation for $F(z)$.
It reads $\theta^dF=z(\theta+1/2)^dF$,
where $\theta$ is the derivation $z{d\over dz}$. A simple verification shows
that if $F(z)$ is any solution of this differential equation then so is $z^{-1/2}F(1/z)$.
The actual proof of Theorem \ref{main} is completely elementary, but at the end of
the proof we sketch the role of the symmetry in the background.

\section{Proofs}
We start with a few well-known elementary congruences.
\begin{lemma}\label{babbage}
For any odd prime $p$ and any integers $0<b\le a$ we have
$${ap\choose bp}\is{a\choose b}\mod{p^2}.$$
\end{lemma}

The theorem was proven by Babbage in 1819, \cite{babbage}.
In 1862 Wolstenholme \cite{wolstenholme} 
showed that this congruence holds modulo $p^3$
for all primes $p\ge5$.

\begin{proof}
Observe that
$${ap\choose bp}=\prod_{k=1}^{(a-b)p}{k+bp\over k}.$$
Split the product into factors with $p|k$ (and write $k=lp$) and factors
where $k$ is not divisible by $p$. We get
$${ap\choose bp}=\prod_{l=1}^{a-b}{l+b\over l}\prod_{k=1\atop (k,p)=1}^{(a-b)p}
\left(1+{bp\over k}\right),$$
where the second product is restricted to $k\not\is0\mod{p}$.
The first factor equals ${a\choose b}$, the second is modulo $p^2$ equal
to
$$1+\sum_{k=1\atop(k,p)=1}^{(a-b)p}{bp\over k}.$$
The well-known fact that $\sum_{k=1}^{p-1}1/k\is0\mod{p}$ implies that 
the second product is $1\mod{p^2}$. This proves our assertion.
\end{proof}

\begin{lemma}\label{eisenstein}
Let $\gamma=(4^{p-1}-1)/p$. Then
$$\sum_{j=1}^{p-1}{(-1)^{j-1}\over j}\is \gamma\mod{p}.$$
\end{lemma}

This lemma occurs in the work of Eisenstein \cite{Eisenstein}.

\begin{proof}
First notice that
$${4^{p-1}-1\over p}={1\over 4p}(4^p-4)={2^p-2\over p}{2^p+2\over 4}.$$
By Fermat the last factor is $1\mod{p}$ and we get that
$${4^{p-1}-1\over p}\is{2^p-2\over p}\mod{p}.$$
We compute the latter modulo $p$. 
$${1\over p}(2^p-2)={1\over p}\sum_{k=1}^{p-1}{p\choose k}
=\sum_{k=1}^{p-1}{1\over k}{p-1\choose k-1}.$$
The number ${p-1\choose k-1}$ is the coefficient of $x^{k-1}$ in
$$(1+x)^{p-1}\is {x^p+1\over x+1}\is 1-x+x^2-x^3+\cdots+x^{p-1}\mod{p}.$$
Hence ${p-1\choose k-1}\is (-1)^{k-1}\mod{p}$ and thus our congruence
follows. 
\end{proof}

\begin{lemma}\label{symmetry}
Define $\alpha_r={(1/2)_r\over r!}$. Then for any odd prime $p$ and any
integer $0\le r<p/2$ we have
$$\alpha_{{p-1\over2}-r}\is(-1)^{p-1\over2}\alpha_r\mod{p}.$$
\end{lemma}

\begin{proof}
Notice that 
$$\alpha_r\is{(1/2)_r\over r!}\is {(1/2-p/2)_r\over r!}\is(-1)^r{(p-1)/2\choose r}
\mod{p}.$$
The symmetry is now immediate from the last expression.
\end{proof}

A direct corollary is the following. 

\begin{corollary}\label{Fzero}
Suppose $p\is3\mod{4}$.
Then $F_p(-\epsilon_p)\is0\mod{p}$. 
\end{corollary}

\begin{proof}
Notice that
\begin{eqnarray*}
F_p(-\epsilon_p)&=&\sum_{r=0}^{(p-1)/2}\alpha_r^d(-\epsilon_p)^r\\
&\is& (-1)^{d(p-1)/2}\sum_{r=0}^{(p-1)/2}\alpha_{{p-1\over2}-r}^d(-\epsilon_p)^r\mod{p}\\
&\is& (-1)^{d(p-1)/2}(-\epsilon_p)^{{p-1\over2}}
\sum_{r=0}^{(p-1)/2}\alpha_{r}^d(-\epsilon_p)^r\mod{p}\\
&\is& -F_p(-\epsilon_p)\mod{p},
\end{eqnarray*}
which implies our assertion.
\end{proof}

\begin{lemma}\label{split}
Let $p$ be an odd prime and $r,r',t$ integers $\ge0$ with $r=pr'+t$
and $t<p$. Let $\alpha_r$ be as in the previous lemma and 
$\gamma=(4^{p-1}-1)/p$. If $p/2<t$, then $p$ divides $\alpha_r$ and
if $t<p/2$ we have
$$\alpha_r\is \alpha_{r'}\alpha_t\left(1-\gamma pr'+
2pr'\sum_{j=1}^{2t}{(-1)^{j-1}\over j}\right)\mod{p^2}.$$
\end{lemma}

Modulo $p$ the congruence reads $\alpha_r\is\alpha_{r'}\alpha_t\mod{p}$. This is
known as the Lucas-property for $\alpha_r$. 

\begin{proof}
Instead of $\alpha_r$ we start with ${2r\choose r}$. Notice that
$${2r\choose r}={2pr'\choose pr'}{\prod_{k=1}^{2t}(k+2pr')\over
\prod_{k=1}^t(k+pr')^2}.$$
Note that if $t>p/2$ the product in the numerator contains the factor
$p+2pr'$ and is therefore divisible by $p$. Suppose from now on that $t<p/2$.

Consider the equation modulo $p^2$. We apply Lemma \ref{babbage} to the
binomial coefficient and get ${2r'\choose r'}$. The product over $k$ becomes
${2t\choose t}$ times
$$1+2pr'\left(\sum_{k=1}^{2t}{1\over k}-\sum_{k=1}^t{1\over k}\right)\mod{p^2}.$$
Notice also that
$$\sum_{k=1}^{2t}{1\over k}-\sum_{k=1}^t{1\over k}=
\sum_{k=1}^{2t}{(-1)^{k-1}\over k}.$$
Finally use the relation ${2r\choose r}=4^r\alpha_r$. Putting everything together
we find that
$$\alpha_r\is\alpha_{r'}\alpha_t 4^{r'(1-p)}\left(
1+2pr'\sum_{k=1}^{2t}{(-1)^{k-1}\over k}\right)\mod{p^2}.$$
Using $4^{r'(1-p)}\is1-pr'\gamma\mod{p}$ yields our assertion.
\end{proof}

{\it Proof} of Theorem \ref{main}.

In view of congruences (\ref{cauchy}) it suffices to prove that 
$F_{p^s}(\epsilon_p)\is F_{p}(\epsilon_p)F_{p^{s-1}}(\epsilon_p)\mod{p^2}$
for $s=2$, but we will do it for all $s\ge2$. Use the notation $\alpha_r
={(1/2)_r\over r!}$ and Lemma~\ref{split} to find

$$
F_{p^s}(z)=\sum_{r'=0}^{p^{s-1}-1}\sum_{t=0}^{(p-1)/2}
(\alpha_{r'}\alpha_t)^dz^{pr'+t}
\left(1-\gamma dpr'+2dpr'\sum_{k=1}^{2t}{(-1)^{k-1}\over k}\right)\mod{p^2}.
$$

The terms with $t>p/2$ do not occur since $\alpha_r^d\is0\mod{p^2}$
whenever $t>p/2$. This gives
$$F_{p^s}(z)\is F_p(z)F_{p^{s-1}}(z^p)+pd\left(G_1(z)-\gamma F_p(z)\right)
\sum_{r'=0}^{p^{s-1}-1}r'z^{pr'}
\alpha_{r'}^d\mod{p^2}$$
where
$$G_1(z)=
2\sum_{t=0}^{(p-1)/2}\left(\sum_{k=1}^{2t}{(-1)^{k-1}\over k}\right)
\alpha_t^d z^t.$$
In order to arrive at our result we set $z=\epsilon_p$ and show that
$G_1(\epsilon_p)\is\gamma F_p(\epsilon_p)\mod{p}$.
Consider $G_1(\epsilon_p)=2\Sigma=\Sigma+\Sigma$ as a sum of two (equal) sums over $t$. 
In one of these we replace $t$ by $(p-1)/2-t$ and obtain
$$\sum_{t=0}^{(p-1)/2}\left(\sum_{k=1}^{p-1-2t}{(-1)^{k-1}\over k}\right)
\alpha_{(p-1)/2-t}^d\epsilon_p^{(p-1)/2-t}.$$
Apply Lemma \ref{symmetry} and replace $k$ in the inner summation by
$p-k$. We get
$$\sum_{t=0}^{(p-1)/2}\left(\sum_{k=2t+1}^{p-1}{(-1)^{-p+k-1}\over p-k}\right)
\alpha_{t}^d\epsilon_p^{t}\mod{p}.$$
This equals
$$\sum_{t=0}^{(p-1)/2}\left(\sum_{k=2t+1}^{p-1}{(-1)^{k-1}\over k}\right)
\alpha_{t}^d\epsilon_p^{t}\mod{p}$$
Thus we obtain after addition of $\Sigma$,
$$G_1(\epsilon_p)\is\sum_{t=0}^{(p-1)/2}\left(\sum_{k=1}^{p-1}{(-1)^{k-1}\over k}\right)
\alpha_{t}^d\epsilon_p^{t}\is \left(\sum_{k=1}^{p-1}{(-1)^{k-1}\over k}\right)
F_p(\epsilon_p)\mod{p^2}.$$
Application of Lemma \ref{eisenstein} yields the desired result.
\medskip

\hfill$\Box$\medskip

\section{The underlying mechanism}
The proof of our main result uses a symmetry of the
polynomials $F_p(z),G_1(z)$ modulo $p$.
We show here how this is forced by the symmetry of the hypergeometric equation.
One easily sees that $F_p(z)\mod{p}$ is the unique polynomial of degree $<p/2$ which
satifies our hypergeometric differential equation modulo $p$ and which has
constant term $1$. 
Furthermore, $F_p(z)\log z+G_p(z)$ is another solution modulo $p$. By the symmetry
of our equation $z^{(p-1)/2}F_p(1/z)$ is also a polynomial solution modulo $p$. 
Hence, by uniqueness of $F_p$, $z^{(p-1)/2}F_p(1/z)\is\lambda F_p(z)\mod{p}$
for some $\lambda$. To determine $\lambda$ we set $z=\epsilon_p$. Then
$\epsilon_pF(\epsilon_p)=\lambda F(\epsilon_p)$. Since $F(\epsilon_p)$ is a $p$-adic
unit by assumption we conclude that $\lambda=\epsilon_p$.
Hence $F_p(z)$ is a reciprocal or anti-reciprocal polynomial. We observe that
$z^{(p-1)/2}F_p(1/z)\log(1/z)+z^{(p-1)/2}G_p(1/z)$ is also a mod $p$ solution.
Multiply by $\epsilon_p$ and add $F_p(z)\log z+G_p(z)$. We find the new solution
$G_p(z)+\epsilon_pz^{(p-1)/2}G_p(1/z)$ which is a polynomial
solution. Hence it equals $\mu F_p(z)$ for some $\mu$. To find the value of $\mu$
we set $z=0$. The constant term of $G_p(z)$ is $0$ and the constant term of 
$\epsilon_pz^{(p-1)/2}G_p(1/z)$ is the leading term of $\epsilon_pG_p(z)$, which is 
$2\sum_{j=1}^{p-1}{(-1)^{j-1}\over j}$, hence
$2\gamma$ by Lemma \ref{eisenstein}.
Using $F_p(0)=1$ we conclude that $\mu=2\gamma$. Now set $z=\epsilon_p$ in 
$$\epsilon_pz^{(p-1)/2}G_p(1/z)+G_p(z)\is 2\gamma F_p(z)\mod{p}$$
and we obtain that $G_p(\epsilon_p)=\gamma F_p(\epsilon_p)$,
the key step in the proof of our theorem.

\section{Hypergeometric motives}\label{motive}
In this section we explain the nature of the unit root $f_p(z_0)$ via 
finite hypergeometric sums and their $\zeta$-functions. For any $q=p^k$
we consider a generator $\omega$ of the multiplicative characters on
$\bbbf_q^{\times}$. Then we define the Gauss-sum
$$g_q(\omega^k)=\sum_{x\in\bbbf_q^{\times}}\omega(x)^k\zeta_p^{\tr(x)},$$
where $\tr:\bbbf_q\to\bbbf_p=\bbbz/p\bbbz$ is the trace map 
and $\zeta_p$ is a primitive $p$-th root of unity.
Let $\phi$ be the unique character of
order $2$. Let $t\in\bbbf_q^{\times}$ and define
$$H_q(t)={(-1)^d\over1-q}\sum_{m=0}^{q-2}\left({g_q(\phi\omega^m)g_q(\omega^{-m})
\over g_q(\phi)}\right)^d\omega((-1)^dt)^m.$$
It turns out that the values are rational
integers which are independent of the choice of $\omega$ and $\zeta_p$.
Such functions were introduced by John Greene and independently Nick Katz
by the end of the 1980's. According to Katz these sums are traces of
the Frobenius operator on $l$-adic cohomology associated to the hypergeometric
differential equation. More concretely, hypergeometric sums show up in
point counting results on algebraic varieties over finite fields. The relevant
example for us is the following.

\begin{theorem}
Let $q$ be an odd prime power, $t\in\bbbf_q^{\times}$ and $d\ge2$ an integer.
Then the number of points with coordinates in $\bbbf_q^{\times}$ on
the hypersurface
$$X_t:\ \prod_{i=1}^d(x_i+2+x_i^{-1})=4^dt^{-1}$$
is given by
$$
{(q-2)^d-(-1)^d\over q-1}-(-1)^d H_q(t).
$$
\end{theorem}

\begin{proof}
This is a consequence of Theorem \cite[Thm 6.1]{BCM}. Since our hypergeometric
parameters are just $1/2$ and $1$ we are in a special situation where the
parameters $a_i$ from \cite[Thm 6.1]{BCM} read $(2,\ldots,2,-1,\ldots,-1)$
with $d$ repetitions of $2$ and $2d$ repetitions of $-1$. The corresponding
variety is given by the intersection of the following varieties
in $(\bbbp^2)^d$,
$$u_1+v_1+w_1=u_2+v_2+w_2=\cdots =u_d+v_d+w_d=0,\quad \lambda \prod_{i=1}^d
u_i^2=\prod_{i=1}^dv_iw_i.$$
Elimination of the $u_i$ gives us $\lambda\prod_{i=1}^d(v_i+w_i)^2=
\prod_{i=1}^dv_iw_i$. Then set $x_i=v_i/w_i$ and $\lambda=t/4^d$
to get the equation of our assertion.
Theorem \cite[Thm 6.1]{BCM} gives the point count with invertible
coordinates in $\bbbf_q$ as
$${(q-2)^d\over q-1}+{1\over q^d(q-1)}\sum_{m=1}^{q-2}g_q(\omega^{2m})^d g_q(\omega^{-m})^{2d}
\omega(\lambda)^m.$$
Use the Hasse-Davenport relation $g_q(2m)=\omega(4)^m g_q(\omega^m)g_q(\phi\omega^m)/g_q(\phi)$
and $g_q(m)g_q(-m)=(-1)^mq$ to get
\begin{eqnarray*}
&&{(q-2)^d\over q-1}+{1\over q-1}\sum_{m=1}^{q-2}\left({g_q(\phi\omega^m)g_q(\omega^{-m})
\over g_q(\phi)}\right)^d
\omega((-4)^d\lambda)^m\\
&=&{(q-2)^d-(-1)^d\over q-1}-(-1)^d H_q(4^d\lambda)
\end{eqnarray*}
We find our desired point count after replacing $\lambda$ by $t/4^d$. 
\end{proof}

We now compute $\zeta$-function associated to the values of $H_q(t)$
(with $t\in\bbbf_p^{\times}$) in the
usual way,
$$Z_p(t,T)=\exp\left({H_{p^s}(t)\over s}T^s\right),$$
which turns out to be a polynomial in $\bbbz[T]$ of degree $d$ when $t\ne1$.
When $t=1$ and $d$ odd the degree is $d-1$, when $t=1$ and $d$ even $Z_p(1,T)$
is a polynomial of degree $d-2$ divided by a factor $1-p^{-1+d/2}T$. 
We shall simply take the $d-2$-degree polynomial for $Z_p(1,T)$ in this case.
 
Here we are not able to prove all this, but we simply mention some folklore
results and conjectures which make up a large body of a project on hypergeometric
motives by F.Rodriguez-Villegas, D.Roberts and M.Watkins. The latter has implemented
the computations in Magma. This is now an impressive library to compute the 
polynomials $Z_p(T)$, and also to manipulate the global $L$-series that contain
the $Z_p(p^{-s})$ as local Euler factors. In addition K.Kedlaya has recently 
announced a Sage-implementation (largely a port of the Magma-implementation)
which also calculates the $Z_p(T)$ for us.

We use some of these calculations to illustrate the background to
the supercongruences and the origin of the unit-root $f_p(z_0)$. The polynomial
$Z_p(t,T)$ can be factored as $\prod_i(1-\mu_iT)$ where the $\mu_i$ are algebraic
and all have the same absolute value $p^{(d-1)/2}$ according to the Weil-conjectures.
The exponent $d-1$ is called the weight of the $\zeta$-factor $Z_p(t,T)$. 
By abuse of language we shall call the $\mu_i$ the zeros of $Z_p(t,T)$. 
The idea is now that if $f_p(z_0)$ is a $p$-adic unit, the polynomial
$Z_p(z_0,T)$ has a unique $p$-adic zero which is a unit, namely $f_p(z_0)$. 
Here are some examples.
\medskip

When $d=4$ and $z_0=1$ we get $Z_p(1,T)=1-a_pT+p^3T^2$ where $a_p$ is the $p$-th
coefficient of $\eta(2\tau)^4\eta(4\tau)^4$. It is clear that when this polynomial
has a unit root $f_p(1)$, the Newton polygon has $p$-adic slopes $0,3$.
Hence $f_p(1)\is a_p\mod{p^3}$. The missing slopes $1,2$ may account for the
occurrence of a supercongruence mod $p^3$.
\medskip

When $d=6$ and $z_0=1$ we get $Z_p(1,T)=(1-pa_pT+p^5T^2)(1-b_pT+p^5T^2)$, 
where $a_p$ is as above and $b_p$ the $p$-th coefficients of the newform
in $S_6(8,\chi_0)$. The Newton slopes of the first one are $1,4$ (if $a_p$
is a unit) and $0,5$ for the second (if $b_p$ is a unit). This shows that
$f_p(1)\is b_p\mod{p^5}$ and one might also consider this as an explanation
for the conjectural supercongruence modulo $p^5$. 
\medskip

In general, when $d$ is even and $z_0=1$, we expect a factorization
$Z_p(1,T)=U_p(T)V_p(T)$ into two factors in $\bbbz[T]$. The degrees 
of $U_p,V_p$ are $-1+d/2,-1+d/2$ when $d=2\mod{4}$ and $-2+d/2,d/2$
if $d\is0\mod{4}$. The factor $U_p$ has one Newton slope $1$ and the others higher. 
The factor $V_p$, when $f_p(1)$ is a unit, has Newton slopes $0,2$ and higher.
So, in a way the factorization of $Z_p(1,T)$ separates the slope $1$ from the
slopes $0,2,\ldots$. Naturally $f_p(1)$ is the unit root zero of $V_p$.
The separation of the slopes may be seen as an explanation of the supercongruences
from Theorem \ref{main}. Speculations of this type were first made by
Dave Roberts and Fernando Rodriguez-Villegas in their preprint
\cite{RobVillegas}. Instead of speaking about Newton slopes they consider
Hodge levels in the cohomology of a hypergeometric motive. 

Finally we record a few factorizations of $Z_p(-1,T)$ when $d$ is odd. This
is a case where factorizations are abundant.
\medskip

When $d=3$ we get 
$$Z_p(-1,T)=(1-pT)(1-c_pT+\chi(-8)p^2T^2).$$
Here $c_p$ is the $p$-th coefficient of the modular form
$\eta(\tau)^2\eta(2\tau)\eta(4\tau)\eta(8\tau)^2$ and is
related to the case $d=3$ in Conjecture \ref{conj3}.
\medskip

When $d=5$ we get 
$$Z(-1,T)=(1-\gamma_p p^2T)(1-pc_pT+p^4T^2)(1-d_pT+p^4T^2),$$
where $d_p$ is the coefficient defined in Conjecture \ref{conj3} and
$c_p$ the $p$-th coefficient of $\eta(4\tau)^6$. The coefficient $\gamma_p$
is $-1$ if $p\is5\mod{8}$ and $1$ otherwise. 
\medskip

When $d=7$ we get
$$Z_p(-1,T)=(1-p^3T)(1-pa_pT+p^6T^2)Q_4(T),$$
where $Q_4$ is a factor of degree $4$.
Here $a_p=\leg{-4}{p}(\phi_p^2-2p^2)$ where $\phi_p$ is the $p$-th coefficient of the form
in $S_3(32,\chi(-4))$ that begins with
$$q+4iq^3+2q^5-8iq^7-7q^9-4iq^{11}-14q^{13}+8iq^{15}+18q^{17}-12iq^{19}+32q^{21}
+40iq^{23}+\cdots$$

Moreover, when $p\is3,5\mod{8}$ the polynomial $Q_4$ factors into $1-p^6T^2$
times a quadratic factor $1-\gamma_pT+p^6T^2$. However, this does not give us anything
stronger than mod $p^2$ congruences. We are indebted to Dave Roberts for the
identification of the modular form.

\end{document}